\documentclass[10pt]{amsart}

\usepackage{amsmath, amsthm, amssymb, graphicx, enumitem}

\newcounter{citedtheorems}

\numberwithin{equation}{section}

\newtheorem{defn}{Definition}[section]
\newtheorem{theorem}[defn]{Theorem}
\newtheorem*{theorem-m}{Theorem \ref{t:rg}}
\newtheorem*{thm-m}{Main Theorem}
\newtheorem*{theorem-abs1}{Theorem \ref{ind-theorem}}
\newtheorem*{theorem-abs2}{Theorem \ref{a23}}
\newtheorem*{theorem-abs3}{Theorem \ref{ind-new}}
\newtheorem*{theorem-abs4}{Theorem \ref{m1}}
\newtheorem{thm-lit}[citedtheorems]{Theorem}
\newtheorem{defn-lit}[citedtheorems]{Definition}
\newtheorem{fact-lit}[citedtheorems]{Fact}
\newtheorem{fact}[defn]{Fact}
\newtheorem{cor}[defn]{Corollary}
\newtheorem{defn-claim}[defn]{Definition/Claim}

\newtheorem{conv}[defn]{Convention}
\newtheorem{claim}[defn]{Claim}

\newtheorem{obs}[defn]{Observation}
\newtheorem{rmk}[defn]{Remark}

\newtheorem{qst}[defn]{Question}

\newcommand{\los}{\L os }
\newcommand{\br}{\vspace{2mm}}
\newcommand{\step}{\br\noindent\emph}
\newcommand{\eff}{\mathcal{F}}
\newcommand{\gee}{\mathcal{G}}
\newcommand{\lp}{\langle}
\newcommand{\rp}{\rangle}

\newcommand{\kleq}{\trianglelefteq}

\newcommand{\rn}{\operatorname{Range}}

\newcommand{\dom}{\operatorname{Dom}}

\newcommand{\mct}{\mathcal{T}}

\newcommand{\de}{\mathcal{D}}

\newcommand{\fss}{{\mathcal{P}}_{\aleph_0}}

\newcommand{\trv}{\mathbf{t}} 

\newcommand{\jj}{\mathbf{j}}
\newcommand{\mcp}{\mathcal{P}}
\newcommand{\xpu}{{\overline{x}}_{\mcp(u)}}

\newcommand{\ba}{\mathfrak{B}}

\newcommand{\lao}{[\lambda]^{<\aleph_0}}
\newcommand{\lba}{\Lambda_{\ba, \overline{a}}}

\newcommand{\rstr}{\upharpoonright}

\newcommand{\vp}{\varphi}
\newcommand{\lcf}{\operatorname{lcf}}

\newcommand{\mcc}{\mathcal{C}}

\newcommand{\ma}{\mathbf{a}}
\newcommand{\mb}{\mathbf{b}}

\newcommand{\fin}{\operatorname{FI}}

\newcommand{\trg}{T_{\mathbf{rg}}}
\newcommand{\dm}{\operatorname{dom}}
\newcommand{\ap}{\operatorname{AP}}
\newcommand{\mca}{\mathcal{A}}
\newcommand{\fil}{\operatorname{fil}}

\newcommand{\xp}{\mathfrak{p}}

\newcommand{\xt}{\mathfrak{t}}

\newcommand{\tlf}{\trianglelefteq}

\title[Saturating the random graph...]{Saturating the random graph with an independent family of small range}

\author{M. Malliaris and S. Shelah}\thanks{\emph{Thanks:}
Malliaris was partially supported by NSF grant DMS-1001666 and by a G\"odel fellowship.
Shelah was partially supported by Israel Science Foundation grant 1053/11. 
This is paper number 1009 in Shelah's list of publications.}

\address{Department of Mathematics, University of Chicago, 5734 S. University Avenue, Chicago, IL 60637, USA and
Einstein Institute of Mathematics, Edmond J. Safra Campus, Givat Ram,
The Hebrew University of Jerusalem, Jerusalem, 91904, Israel}
\email{mem@math.uchicago.edu}
\address{Einstein Institute of Mathematics, Edmond J. Safra Campus, Givat Ram, The Hebrew
University of Jerusalem, Jerusalem, 91904, Israel, and Department of Mathematics,
Hill Center - Busch Campus, Rutgers, The State University of New Jersey, 110
Frelinghuysen Road, Piscataway, NJ 08854-8019 USA}
\email{shelah@math.huji.ac.il}
\urladdr{http://shelah.logic.at}

\begin{document}



\begin{abstract}
Motivated by Keisler's order, 
a far-reaching program of understanding basic model-theoretic structure through the lens
of regular ultrapowers, 
we prove that for a class of regular filters $\de$ on $I$, $|I| = \lambda > \aleph_0$, 
the fact that $\mcp(I)/\de$ has little freedom (as measured by the fact that any maximal antichain is of 
size $<\lambda$, or even countable) 
does not prevent extending $\de$ to an ultrafilter $\de_1$ on $I$ which saturates ultrapowers of the random graph.
``Saturates'' means that $M^I/\de_1$ is $\lambda^+$-saturated whenever $M \models \trg$.
This was known to be true for stable theories, and false for non-simple and non-low theories. This result and the techniques
introduced in the proof have catalyzed the authors' subsequent work on Keisler's order for simple unstable theories. 
The introduction, which includes a part written for model theorists and a part written for set theorists, discusses our current program and related results. 
\end{abstract}

\maketitle

\hfill \emph{Dedicated to Jouko V\"{a}\"{a}n\"{a}nen on the occasion of his 60th birthday.}

\br

\section{Introduction} 

Keisler's order is a long-standing (and far-reaching) program for comparing the complexity of unstable theories, 
proposed in Keisler 1967 \cite{keisler}. 
The measure of complexity, roughly speaking, is the relative difficulty of producing saturated regular ultrapowers. 
The order $\tlf$ already has significant connections to classification theory, and we believe the present investigations
will shed further light on the structure of simple unstable theories. 
In \S \ref{intro:order} below, we define the order and survey what was known.  
In \S \ref{intro:pres} we present this paper's main result, Theorem \ref{t:rg}, and explain the relevance for simple theories.  
In \S \ref{intro:tr}, we describe the known points of contact between regular ultrafilters and theories. 
In \S \ref{intro:st}, we discuss several theorems of set theory which have come from this program. 

\subsection{Keisler's order and model theory} \label{intro:order} 
This subsection, written primarily for model theorists, aims to explain Keisler's order and our recent work.

For transparency, all languages are countable and all theories are complete. 
We say that the ultrapower $M^\lambda/\de$
is \emph{regular} when $\de$ is a regular ultrafilter on $\lambda$, Definition \ref{regular} below.
A key property of regularity is that:
\begin{fact} \label{eq:fact}
Let $M \equiv N$ in a countable signature, $\lambda \geq \aleph_0$, $\de$ a regular ultrafilter on $\lambda$.
Then $M^\lambda/\de$ is $\lambda^+$-saturated if and only if $N^\lambda/\de$ is $\lambda^+$-saturated.
\end{fact}

\subsubsection{What is Keisler's order and why is it model-theoretically interesting?}.

Questions of saturation have long been central to model theory. Morley's theorem is a fundamental example.
Suppose we anachronistically define the \emph{\los order} on complete countable theories:

\begin{itemize}
\item $T_1 \tlf^{L}_\lambda T_2$ if 
\begin{itemize}
\item (for all $M_2 \models T_2$, $|M_2| = \lambda$, $M_2$ is saturated) implies
\item (for all $M_1 \models T_1$, $|M_1| = \lambda$, $M_1$ is saturated)
\end{itemize}
\item $T_1 \tlf^{L} T_2$ if $T_1 \tlf^{L}_\lambda T_2$ for all $\lambda > \aleph_0$ 
\end{itemize}

Then Morley's theorem shows that $\tlf^L$ has a minimum class, a maximum class, and no other classes, i.e.
either all uncountable models of some given countable theory are saturated, or else the theory has some
unsaturated model of every uncountable size. 

Keisler's order may be thought of as generalizing this hypothetical ``\los order'' in the following powerful way:
rather than considering all uncountable models, we consider only regular ultrapowers (so ``saturated'' becomes ``$\lambda^+$-saturated''). 
This reveals a richer field of comparison: we compare not by cardinality, but by provenance. Each $M_1^\lambda/\de \models T_1$ 
is naturally compared to models $M_2^\lambda/\de$ of $T_2$, built using the same ultrafilter $\de$ on $\lambda$.
More precisely:

\begin{defn} \label{keisler-order} \emph{(Keisler \cite{keisler})} 
Let $T_1, T_2$ be complete countable first-order theories. 
\begin{enumerate}
 \item Let $\de$ be a regular ultrafilter on $\lambda$. Write \emph{$T_1 \kleq_\de T_2$} 
 when (for all $M_2 \models T_2$, $M^{\lambda}_2/\de$
is $\lambda^+$-saturated) implies (for all $M_1 \models T_1$, $M^{\lambda}_1/\de$ is $\lambda^+$-saturated). 
\item Write $T_1 \tlf_\lambda T_2$ if for any regular ultrafilter $\de$ on $\lambda$, 
 $T_1 \tlf_\de T_2$.  
\item \emph{(Keisler's order)} 
Write $T_1 \kleq T_2$ if for all infinite $\lambda$, $T_1 \kleq_\lambda T_2$.
\end{enumerate}
\end{defn}

By Fact \ref{eq:fact}, $\tlf$ is a pre-order on theories, usually thought of as a partial order on the equivalence classes. 
[Note that by \ref{eq:fact}, $T_1 \tlf_\de T_2$ is equivalent to ``for all $M_1 \models T_1$, $M_2 \models T_2$,
($M^{\lambda}_2/\de$ is $\lambda^+$-saturated) implies ( $M^{\lambda}_1/\de$ is $\lambda^+$-saturated).''] 

Keisler proved that $\tlf$ had at least a minimum and maximum class, and asked:

\begin{qst} \label{q:k} \emph{(Keisler 1967)}
Determine the structure of Keisler's order. 
\end{qst}

Surprising early work of Shelah established the model theoretic significance of $\tlf$ (note the independent appearance of
dividing lines from classification theory).

\begin{thm-lit} \emph{(Shelah 1978 \cite{Sh:a} Chapter VI)} \label{thm:s}
\[ \mct_1 < \mct_2 < \cdots \cdots ? \cdots \cdots \leq \mct_{max} \]
where $\mct_1 \cup \mct_2$ is precisely the class of countable stable theories, and:
\begin{enumerate}
\item $T_1$, the minimum class, is the set of all $T$ without the finite cover property. 
\item $T_2$, the next largest class, is the set of all stable $T$ with the f.c.p.
 \item There is a maximum class $\mct_{max}$, containing all linear orders (i.e. SOP, in fact $SOP_3$ \cite{Sh500}), 
 however, its model-theoretic identity is not known. 
\end{enumerate}
\end{thm-lit}

For many years there was little progress, and the unstable case appeared relatively intractable despite the
flourishing of unstable model theory. 

Very recently, work of Malliaris and Shelah has led to considerable advances in our
understandings of how ultrafilters and theories interact (Malliaris \cite{mm-thesis}-\cite{mm5}, Malliaris and Shelah \cite{MiSh:996}-\cite{MiSh:999}).
For an account of some main developments, see the introductory sections of \cite{MiSh:996} as well as \cite{MiSh:999}. 
Most of the seismic shifts are underground and not yet visible as divisions in the order. Still, one can update the diagram:
 \[ \mct_1 < \mct_2 < \mct_{x}  \cdots ? \cdots < \cdots ? \cdots \mct_{nl} \leq \cdots ? \cdots \leq \mct_{max} \]
$\mct_x$ denotes the minimum unstable class (not yet characterized, but contains the random graph). $\mct_{nl}$ denotes
the simple non low theories, which are not known to be an equivalence class but are all strictly below $\mct_x$. A more 
finely drawn diagram would also show that there is a minimum $TP_2$ theory, 
and that at limit $\lambda$ $SOP_2$ suffices for maximality, see \S \ref{intro:st}. 
Note that it is not yet known whether the order is finite, or linear.

\subsection{Prolegomena to our main theorem} \label{intro:pres}
Earlier this year, building on our prior work described above, 
we had shown that the question of saturation of ultrapowers could be substantially recast
in terms of a two-stage approach, involving a more set-theoretic stage [constructing a so-called \emph{excellent} filter $\de$ on $I$ 
admitting a surjective homomorphism $\jj: \mcp(I) \rightarrow \ba$ onto a specified Boolean algebra, with $\jj^{-1}(\{ 1_\ba \}) = \de$]
followed by a more model-theoretic stage [constructing a so-called \emph{moral} ultrafilter on the Boolean algebra $\ba$, a step defined
in terms of ``possibility patterns'' of formula incidence represented in the theory]. 

This advance allowed us to give the
first ZFC $\tlf$-dividing line among the unstable theories, and the first dividing line since 1978: separating the Keisler-minimum unstable theory, 
the random graph $\trg$, from all non-low or non-simple theories.

\begin{thm-lit} \emph{(Malliaris and Shelah \cite{MiSh:999})} \label{999-theorem}
Suppose $\lambda, \mu$ are given with $\mu < \lambda \leq 2^\mu$. Then 
there is a regular ultrafilter $\de$ on $\lambda$ which saturates ultrapowers of
all countable stable theories and of the random graph, but fails to saturate ultrapowers of any non-low or non-simple theory. 
\end{thm-lit}

\begin{conv}
We say that a regular ultrafilter $\de_1$ on $\lambda$ ``saturates ultrapowers of $T$'' to mean that whenever $M \models T$, $M^\lambda/\de$ 
is $\lambda^+$-saturated.
\end{conv}

Naturally this result raised many further questions. Chief among them was the role of $\mu < \lambda \leq 2^\mu$, which appeared to have a strong
model-theoretic motivation by analogy to \cite{Sh93}, 
and in terms of what might be called the ``Engelking-Kar\l owicz property of the random graph:'' in the monster model of the
random graph, if $|B| = \lambda$ and $\mu < \lambda \leq 2^\mu$ then there is some set $A$ such that every nonalgebraic $p \in S(B)$ can be
finitely realized in $A$. 
The role of $\mu$ in the proof of Theorem \ref{999-theorem} was as the size of a maximal antichain of a quotient Boolean algebra 
at the key transfer point in the inductive ultrafilter construction,
i.e., the point where the excellent filter $\de$ has been built. 
When $\mu < \lambda$, this ``lack of freedom'' at the transfer point was sufficient to guarantee the non-saturation result for any 
ultrafilter $\de_1 \supseteq \de$, and as indicated, when $\lambda \leq 2^\mu$, 
model-theoretic considerations made it possible to guarantee saturation of the random graph for \emph{some} $\de_1 \supseteq \de$.  

In this paper we show that, contrary to initial expectations, the situation here is fundamentally different and
the second restriction on $\mu$ is unnecessary. 
Rather it is, within the regime of so-called excellent filters on some arbitrary but fixed $\lambda$, 
\emph{always} possible to extend to an ultrafilter which will saturate the random graph,
even when $\mu = \aleph_0$ at the transfer point: 

\begin{thm-m}
Let $\aleph_0 \leq \mu \leq \lambda = |I|$. Suppose that $\gee_*$ and $\de$ satisfy: 
\begin{enumerate}
\item $\de$ is a regular, excellent filter on $I$
\item $\gee_* \subseteq {^\lambda\mu}$, $|\gee_*| = 2^\lambda$ 
\item $\gee_*$ is a $\de$-independent family of functions
\item $\de$ is maximal subject to (3)
\end{enumerate}
Then for some ultrafilter $\de_1 \supseteq \de$, and all $M \models \trg$, $M^I/\de_1$ is $\lambda^+$-saturated.
\end{thm-m}

Moreover, the proof itself modifies the usual inductive construction of an ultrafilter by means of independent functions by introducing
so-called approximations; this has been significant for our subsequent work.

Note that when $(I, \de, \gee)$ is $(\lambda, \aleph_0)$-good it is always
possible to extend $\de$ to an ultrafilter $\de_1$ which saturates all stable theories. [It suffices to ensure that 
the coinitiality of $\omega$ in $(\omega, <)^I/\de_1$ is $\geq \lambda^+$.] Our theorem here shows that \emph{if} $\de$ is excellent this
is additionally always possible for the random graph.  
It complements our construction in \cite{MiSh:997} of a filter $\de$, necessarily not excellent, \emph{no}
extension of which is able to saturate the random graph.

\br
\subsubsection{What avenues of investigation does this suggest for simple theories?}
The parameter $\mu$ measuring the freedom available in the underlying Boolean algebra during ultrafilter construction, 
i.e. $\mu$ in Definition \ref{good-triples} below, appears significant for ``outside definitions'' of simple theories. 
Given Theorem \ref{t:rg}, the natural question is whether 
membership in the Keisler minimum unstable class is characterized by: 
whenever $\de$ is regular and excellent and $(I, \de, \gee)$ a $(\lambda, \aleph_0)$-good triple, some ultrafilter $\de_1 \supseteq \de$ 
saturates $T$.
By our prior results, $\mu < \lambda$ blocks saturation of non-simple or non-low theories (this can be circumvented
by introducing a complete filter on the quotient Boolean algebra). Thus, the current focus is simple low theories. 

By increasing the range of possible $\mu$, Theorem \ref{t:rg} opens the door to a stratification of simple low theories.
That is, we will want to try to distinguish between classes of theories based on various cardinal invariants of the Boolean algebras
$\mcp(I)/\de$. Our previous result, Theorem \ref{999-theorem} above,
distinguished between the random graph (suffices to have antichains of size $\mu$) and non-low theories (necessary to have antichains of size $\lambda$) when $\mu < \lambda \leq 2^\mu$. So to see finer divisions in the unstable low theories,
one might need to assume failures of GCH. In light of Theorem \ref{t:rg} this is no longer necessary. 

The property $\mu = \aleph_0$ appears connected with a model-theoretic coloring property introduced in 
work in progress of the authors. 

\subsection{Translations between set theory and model theory} \label{intro:tr}
A fundamental part of investigating Keisler's order involves construction of ultrafilters, thus combinatorial
set theory. Isolating ``model-theoretically meaningful properties'' and determining implications and nonimplications
between them gives a useful perspective on ultrafilters. We now include two ``translation'' theorems from \cite{MiSh:996}.
(Some definitions are given in the Appendix.)

\begin{thm-lit} \emph{(Malliaris and Shelah \cite{MiSh:996} Theorem F)} \label{formula-corr}
In the following table, for each of the rows \emph{(1),(3),(5),(6)} the regular ultrafilter $\de$ on $\lambda$ fails to have 
the property in the left column if and only if it omits a type in every formula with the property in the right column.
For rows \emph{(2)} and \emph{(4)}, $\de$ fails to have the property on the left then it omits a type 
in every formula with the property on the right.

\br
\begin{tabular}{lcl}
\textbf{Set theory:} & & \textbf{Model theory:} \\
\br \textbf{properties of filters} & & \textbf{properties of formulas} \\
 \emph{(1)} $\mu(\de) \geq \lambda^+$ & & A. finite cover property \\
 \emph{(2)} $\lcf(\aleph_0, \de) \geq \lambda^+$ & ** & B. order property \\
 \emph{(3)} saturates $T_{rg}$ & & C. independence property \\
 \emph{(4)} flexible, i.e. $\lambda$-flexible	& ** & D. non-low \\
 \emph{(5)} good for equality & & E. $TP_2$ \\
 \emph{(6)} good, i.e. $\lambda^+$-good&  & F. strict order property \\ 
\end{tabular}
\end{thm-lit}

\begin{proof} The characterization of the maximum class via good ultrafilters and the definition of
the f.c.p. are due to Keisler 1967 \cite{keisler}, see \cite{MiSh:996} for details. 
(1)-(2) Shelah 1978 \cite{Sh:a} VI.5. (3) Straightforward by quantifier elimination. 
(4) Malliaris 2009 \cite{mm-thesis}. (5) Malliaris 2010 \cite{mm4}. (6) Shelah 1978 \cite{Sh:a} VI.2.6.
\end{proof}

\begin{thm-lit} \emph{(updated version of Malliaris and Shelah \cite{MiSh:996} Theorem 4.2)} \label{thm:imp}
Assume that $\de$ is a regular ultrafilter on $\lambda$ \emph{(}note that not all of these properties
imply regularity\emph{)}. Then:

$(1) \leftarrow (2) \leftarrow (3) \leftarrow (5) \leftarrow (6)$, with $(1) \not\rightarrow (2)$,  
$(2) \not\rightarrow (3)$,
$(3) \not\rightarrow (5)$, and whether $(5)$ implies $(6)$ is open.
Moreover $(1) \leftarrow (4) \leftarrow (5) \leftarrow (6)$, where $(3) \not\rightarrow (4)$ thus $(2) \not\rightarrow (4)$, $(4) \not\rightarrow (3)$, consistently $(4) \not \rightarrow (5)$, consistently $(4) \not\rightarrow (6)$; and $(4)$ implies $(2)$ is open.
\end{thm-lit}

``Consistently'' throughout Theorem \ref{thm:imp} means assuming a measurable cardinal. One of the surprises of \cite{MiSh:996}-\cite{MiSh:997} was the relevance
of measurable cardinals in constructing regular ultrafilters. This bridges a certain cultural gap between regular filters, typically
used in model theory, and complete ultrafilters, used primarily in set theory.
\S \ref{intro:st} gives some consequences.

\br

\subsection{Set-theoretic theorems and aspects of this program} \label{intro:st}
We now discuss some set-theoretic results of our program, from \cite{MiSh:996} and \cite{MiSh:999}.

First, ``consistently $(4) \not\rightarrow (6)$'' in Theorem \ref{thm:imp} above addressed a question raised in Dow 1975 \cite{dow} 
about whether OK filters (introduced by Kunen, Keisler) are necessarily good. We proved that consistently the gap may be arbitrarily large:

\begin{thm-lit} \emph{(Malliaris and Shelah \cite{MiSh:996} Theorem 6.4)} 
Assume $\kappa > \aleph_0$ is measurable and $2^\kappa \leq \lambda = \lambda^\kappa$.
Then there exists a regular uniform ultrafilter $\de$ on $\lambda$ such that $\de$ is $\lambda$-flexible, thus $\lambda$-OK, 
but not $(2^\kappa)^+$-good.
\end{thm-lit}

Notably, the failure of goodness is ``as strong as possible'' given the construction: 
$\de$ will fail to $(2^\kappa)^+$ saturate the random graph, thus any unstable theory. 
See \cite{MiSh:996} \S 1.2 for an account of this result.
As explained there, this is a natural context in which to study further weakenings of goodness. 
The two papers \cite{MiSh:996}-\cite{MiSh:997} contain several ultrafilter existence theorems.

We now describe a program from our paper \cite{MiSh:998}.  Because of the Keisler-maximality of linear order, but
also for set theoretic reasons (e.g. cardinal invariants of the continuum), it is natural to study
cuts by defining:

\begin{defn} \emph{(\cite{MiSh:998})}
Let $\de$ be a regular ultrafilter on $I$, $|I| = \lambda$. 
Let $\mcc(\de)$ be the set of all $(\kappa_1, \kappa_2)$ such that
$\kappa_1, \kappa_2 \leq \lambda$ are regular and $(\mathbb{N},<)^I/\de$ has a $(\kappa_1, \kappa_2)$-cut. 
\end{defn}

\begin{defn} \emph{(\cite{MiSh:998})}
Say that a regular ultrafilter $\de$ on $\lambda$ has $\lambda^+$-treetops if in any
tree $(T, \leq_T)$ definable in a $\de$-ultrapower, every $\leq_T$-increasing sequence of length $\leq \lambda$
has a $\leq_T$-upper bound.
\end{defn}

Both model-theoretic and set-theoretic considerations pointed to the question of whether
$\lambda^+$-treetops implies $\lambda^+$-good.

We had shown that an ultrafilter $\de$ on $\lambda$ has $\lambda^+$-treetops
if and only if $\mcc(\de)$ has no symmetric cuts, so the question was:

\begin{qst}
Given $\kappa \leq \lambda \implies (\kappa, \kappa) \notin \mcc(\de)$, what are the possible $\mcc(\de)$?
\end{qst}

The solution is the best possible:

\begin{thm-lit} \emph{(Malliaris and Shelah \cite{MiSh:998})}
Suppose that for all $\kappa \leq \lambda$, $(\kappa, \kappa) \notin \mcc(\de)$. 
Then $\mcc(\de) = \emptyset$, so $\de$ is good.
\end{thm-lit}

The framework of \cite{MiSh:998} is somewhat more general; as a result, we obtain two corollaries of this theorem.
First, $SOP_2$ is maximal in Keisler's order. Second, $\xp = \xt$, where $\xp, \xt$ are 
the pseudointersection number and the tower number respectively; this solved the oldest problem on 
cardinal invariants of the continuum.

This concludes the introduction. We now work towards Theorem \ref{t:rg}.

\section{Preliminaries}

Here we define: regular filters, excellent filters, and good triples. 


\begin{defn} \label{regular} \emph{(Regular filters)}
Let $\de$ be a filter on an index set $I$ of cardinality $\lambda$. 
A $\mu$-regularizing family $\{ X_i : i < \mu \}$ is a set such that:
\begin{itemize}
 \item for each $i<\mu$, $X_i \in \de$, and
 \item for any infinite $\sigma \subset \mu$, we have $\bigcap_{i \in\sigma} X_i = \emptyset$
\end{itemize}
Equivalently, for any element $t \in I$, $t$ belongs to only finitely many of the sets $X_i$. 
A filter $\de$ on an index set $I$ of cardinality $\lambda$ is said to be \emph{$\mu$-regular} if it contains
a $\mu$-regularizing family. $\de$ is called \emph{regular} if it is $\lambda$-regular, i.e. $|I|$-regular. 
\end{defn}

Regular filters on $\lambda$ always exist, see \cite{ck73} or \cite{MiSh:999} top of p. 7.

The proof of Theorem \ref{t:rg} builds on a main innovation of Malliaris and Shelah \cite{MiSh:999}, so-called excellent filters: 
\ref{d:excellent} below. Suppose we are given a filter $\de$, a Boolean algebra $\ba$ 
and a surjective homomorphism $\jj: \mcp(I)\rightarrow \ba$ such that $\jj^{-1}(\{ 1_\ba \} ) = \de$. 
Roughly speaking,
if $\de$ is excellent we may assume that $\jj$ is ``accurate'' in the sense that 
if certain distinguished Boolean terms hold on elements $\overline{\mb}$ of $\ba$,
the same terms will hold on some $\jj$-preimage of $\overline{\mb}$ \emph{without} the a priori necessary qualifier ``$\mod \de$.''
[For details and history, see \cite{MiSh:999} \S 2.]
The main example for us here is:

\begin{fact} \label{e:fact} Let $\lambda \geq \aleph_0$ and
let $\de$ be an excellent, i.e. $\lambda^+$-excellent filter on $I$, $|I| = \lambda$. Let 
\[ \overline{A} = \langle A_u : u \in [\lambda]^{<\aleph_0} \rangle \]
be a sequence of elements of $\mcp(I)$ which is multiplicative $\mod \de$, i.e. for each $u, v \in \lao$,
$A_u \cap A_v = A_{u \cup v} \mod \de$.

Then there exists $\overline{B} = \langle B_u : u \in \lao \rangle$ such that:
\begin{enumerate}
\item $u \in \lao \implies B_u \subseteq A_u$
\item $u \in \lao \implies B_u = A_u \mod \de$
\item $u, v \in \lao \implies B_u \cap B_v = B_{u \cup v}$
\end{enumerate}
i.e. $\overline{B}$ refines $\overline{A}$ and is truly multiplicative. 
\end{fact}

\begin{proof}
\cite{MiSh:999} Claim 4.9.
\end{proof}

The effect of excellence is to allow a so-called ```separation of variables,'' \cite{MiSh:999} Theorem 5.11.  
That is, as sketched in \ref{intro:pres}, 
ultrafilter construction can now be done in two stages. First, one builds an excellent filter with a specified quotient $\ba$. 
One can then build ultrafilters directly on $\ba$ which ensure that so-called possibility patterns of a given theory
(a measure of the complexity of incidence in $\vp$-types) have multiplicative refinements. An immediate advantage of this
separation is that it allows us to prove results like Theorem \ref{999-theorem} above. Namely, a ``bottleneck'' is built in to the
construction by arranging for the c.c. of $\mcp(I)/\de$ to be small by the time $\de$ is excellent. This prevents
future saturation of some theories, but not all, and so gives a dividing line.
This analysis also pushes one to understand how these ``possibility patterns'' reflect model-theoretic complexity. 

Though we include the full definition of excellence from \cite{MiSh:999}, 
Theorem \ref{t:rg} will only use Fact \ref{e:fact} above, which the reader may prefer to take as axiomatic.

\begin{defn} \label{c:near}
Let $\ba$ be a Boolean algebra and $\overline{\ma} = \langle \ma_u : u \in \lao \rangle$ be a sequence of 
elements of $\ba$. When $u$ is a finite set, write $\overline{x}_{\mcp(u)} = \langle x_v : v \subseteq u \rangle$
for a sequence of variables indexed by subsets of $u$.
\begin{enumerate}
\item Define 
\begin{align*} 
 N(\overline{a}{\rstr}_{\mcp(u)}) = & \{ \langle a^\prime_v : v \subseteq u \rangle  : ~\mbox{for some $w \subseteq u$}  \\
  & \mbox{we have $a^\prime_v = a_v$ if $v \subseteq w$ and $a^\prime_v = 0_\ba$ otherwise} \} \\
\end{align*}
\item Define $\lba$ to be the set 
\begin{align*}
\{ \sigma(\xpu) : \sigma(\xpu) & ~\mbox{is a Boolean term such that}  \\
& \mbox{ $\ba \models$ ``$\sigma(\overline{a}^\prime) = 0$'' whenever $\overline{a}^\prime \in N(\overline{a})$ }\} 
\end{align*}
\item If $\de$ is a filter on $\ba$ then 
$\Lambda_{\ba, \de, \overline{a}} = \Lambda_{\ba_1, \overline{a}_1}$ where $\ba_1 = \ba/\de$ and 
$\overline{a}_1 = \langle a_v/\de : v \subseteq u \rangle$. 
\br
\item If $\de$ is a filter on a set $I$, then $\de$ determines $I$, so we write
$\Lambda_{\de, \overline{a}}$ for $\Lambda_{\mcp(I), \de, \overline{a}}$.
\end{enumerate}
\end{defn}

\begin{defn} \emph{(Excellent filters, Malliaris and Shelah \cite{MiSh:999} Definition 4.6)} \label{d:excellent} 
Let $\de$ be a filter on the index set $I$. We say that $\de$ is \emph{$\lambda^+$-excellent}
when: if $\overline{A} = \langle A_u : u \in \lao \rangle$ with $u \in \lao \implies A_u \subseteq I$, then
we can find $\overline{B} = \langle B_u : u \in \lao \rangle$ such that:
\begin{enumerate}
\item for each $u \in \lao$, $B_u \subseteq A_u$
\item for each $u \in \lao$, $B_u = A_u \mod \de$
\item \emph{if} $u \in \lao$ and $\sigma \in \Lambda_{\de, \overline{A}|_u}$, so $\sigma(\overline{A}|_{\mcp(u)}) = \emptyset \mod \de$, 
\\ \emph{then} $\sigma(\overline{B}|_{\mcp(u)}) = \emptyset$
\end{enumerate}
We say that $\de$ is $\xi$-excellent when it is $\lambda^+$-excellent for every $\lambda < \xi$. 
\end{defn}


\begin{defn}
Given a filter $\de$ on $\lambda$, we say that a family $\eff$ of functions from $\lambda$ into $\lambda$ is
\emph{independent $\mod \de$} if for every $n<\omega$, distinct $f_0, \dots f_{n-1}$ from $\eff$ and choice of $j_\ell \in \rn(f_\ell)$,
\[ \{ \eta < \lambda ~:~ \mbox{for every $i < n, f_i(\eta) = j_i$} \} \neq \emptyset ~~~ \operatorname{mod} \de \]
\end{defn}

\begin{thm-lit} \label{thm-iff} \emph{(Engelking-Kar\l owicz \cite{ek} Theorem 3, see also Shelah \cite{Sh:c} Theorem A1.5 p. 656)}
For every $\lambda \geq \aleph_0$ there exists a family $\eff$ of size $2^\lambda$ with each $f \in \eff$ from 
$\lambda$ onto $\lambda$ such that $\eff$ is independent
modulo the empty filter \emph{(}alternately, by the filter generated by $\{ \lambda \}\emph{)}$. 
\end{thm-lit}

In particular, such families can be naturally thought of as Boolean algebras, so we introduce some notation:

\begin{defn} \label{d:ba} Denote by
$\ba^1_{\chi, \mu}$ the completion of the Boolean algebra generated by 
 $\{ x_{\alpha, \epsilon} : \alpha < \chi, \epsilon < \mu \}$ freely except for the conditions
$\alpha < \chi \land \epsilon < \zeta < \mu \implies x_{\alpha, \epsilon} \cap x_{\alpha, \zeta} = 0$.
\end{defn}

We follow the literature in using the term ``good triple'' for the following object, despite the name's ambiguity.

\begin{defn} \emph{Good triples (cf. \cite{Sh:c} Chapter VI)} \label{good-triples}
Let $\lambda \geq \kappa \geq \aleph_0$, $|I| = \lambda$, $\de$ a regular filter on $I$, and 
$\gee$ a family of functions from $I$ to $\kappa$.
\begin{enumerate}
\item Let $\fin(\gee) =$ 
\[ \{ h ~: ~h: [\gee]^{<\aleph_0} \rightarrow \kappa ~\mbox{and}~ g \in \dom(g) \implies h(g) \in \rn(g) \} \]
\item Let $\fin_s(\gee) = \{ A_h : h \in \fin(\gee) \}$ where
\[ A_h = \{ t \in I ~:~ g \in \dom(g) \implies g(t) = h(g) \} \]
\item We say that triple $(I, \de, \gee)$ is $(\lambda, \kappa)$-pre-good when $I$, $\de$, $\gee$ are as given, and
for every $h \in \fin(\gee)$ we have that $A_h \neq \emptyset \mod \de$.
\br
\item We say that $(I, \de, \gee)$ is $(\lambda, \kappa)$-good when $\de$ is maximal subject to this condition.
\end{enumerate}
\end{defn}

\begin{fact} \label{good-dense} 
If $(I, \de, \gee)$ is a good triple, then $\fin_s(\gee)$ is dense in $\mcp(I) \mod \de$.
\end{fact}

\begin{obs} \label{obs1}
Let $\de$ be a filter on $I$, $\ba = \ba^1_{\xi, \mu}$ for some $\xi \leq 2^\lambda, \mu \leq \lambda$ and $\jj: \mcp(I) \rightarrow \ba$ a
surjective homomorphism such that $\jj^{-1}( \{ 1_\ba \} ) = \de$. Let $\gee = \gee_\ba = \{ g_\alpha : \alpha < 2^\lambda \} \subseteq {^\lambda \mu}$ be given by $g_\alpha(\epsilon) = \jj^{-1}(x_{\alpha, \epsilon})$. 
Then $(I, \de, \gee)$ is a $(\lambda, \mu)$-good triple. 
\end{obs}

We will use the following existence theorem from \cite{MiSh:999}. 

\begin{thm-lit} \emph{(Existence theorem for excellent filters, \cite{MiSh:999})} \label{t:existence}
{Let $\mu \leq \lambda$, $|I| = \lambda$ and let $\ba$ be a $\mu^+$-c.c. complete Boolean algebra of cardinality $\leq 2^\lambda$. 
Then there exists a regular excellent filter $\de$ on $I$ and a surjective homomorphism $\jj: \mcp(I) \rightarrow \ba$
such that $\jj^{-1}( \{ 1_\ba \} ) = \de$.}
\end{thm-lit}

\begin{cor}
Let $I$, $|I| = \lambda \geq \aleph_0$, and $\mu$ with $\aleph_0 \leq \mu \leq \lambda$ be given.
Let $\de$ be an excellent filter on $I$ given by Theorem \ref{t:existence} in the case where $\ba = \ba^1_{2^\lambda, \mu}$. 
Let $\gee_\ba \subseteq {^\lambda \mu}$ be given by Observation \ref{obs1}.
Then $(I, \de, \gee_\ba)$ is a $(\lambda, \mu)$-good triple.
\end{cor}

\section{Main Theorem} \label{s:main-theorem}

We now prove that it is possible to saturate the theory $\trg$
of the random graph using only functions with range $\aleph_0$. The construction is a natural evolution of
our argument for (an ultrapower version of) the Engelking-Kar\l owicz property in \cite{MiSh:999} Lemma 9.9. 
Here, notably, we modify the usual ``inductive construction via independent families''
to allow a much finer degree of control. 
The calibrations are noted throughout the proof, beginning with
\ref{r:note}. 

\begin{rmk}
Note that in Theorem \ref{t:rg}, possibly $2^\mu << \lambda$;
indeed, possibly $\mu = \aleph_0$ while $\lambda = |I|$ is
arbitrary.
\end{rmk}

\begin{theorem} \label{t:rg}
Suppose that we are given:
\begin{enumerate}
\item $(I, \de, \gee_*)$ is a $(\lambda, \mu)$-good triple
\item $\aleph_0 \leq \mu \leq \lambda = |I|$
\item $\de$ is $\lambda$-regular
\item $\de$ is $\lambda^+$-excellent
\end{enumerate}
Then there is an ultrafilter $\de_1 \supseteq \de$ such that for any $M \models \trg$,
$M^I/\de_1$ is $\lambda^+$-saturated.
\end{theorem}

We shall fix $I$, $\de$, $\gee_*$, $\mu, \lambda$ as in the
statement of Theorem \ref{t:rg} for the remainder of this section.
Hypothesis (4), excellence, is used only at one point, in
Step 13. 

The infrastructure for the proof will be built via several
intermediate claims and definitions. \emph{Note}: When $\vp$ is a formula, we write $\vp^0$ for $\neg \vp$ and $\vp^1$ for $\vp$.

\begin{rmk} \label{r:note}
In contrast to the usual method, we will not complete the
filter at each inductive step to a ``good'' triple. This is a
crucial difference. Rather, we build a series of approximations to
the final ultrafilter. Note that condition (c) on the size of the
approximation is natural, though not needed.

\emph{Background Note}. An account of why realizing types in ultrapowers amounts to 
finding a multiplicative refinement is given in \cite{MiSh:996} \S 1.2, or see \ref{r:good} below.
\end{rmk}

\step{1. Approximations} For $\alpha \leq 2^\lambda$, let
$\ap_\alpha$ be the set of pairs $\ma = (\mca, \gee) = (\mca_\ma,
\gee_\ma)$ such that:
\begin{enumerate}
\item[(a)] $\gee_\ma \subseteq \gee_*$

\item[(b)] $|\gee_\ma| \leq |\alpha| + \lambda$

\item[(c)] $|\mca_\ma| \leq |\alpha| + \lambda$

\item[(d)] $\mca_\ma \subseteq \mcp(I)$

\item[(e)] $(\forall A \in \mca_\ma)(A ~\mbox{is supported by $\gee_\ma$ modulo $\de$})$

\item[(f)] $\emptyset \notin \fil\lp\de \cup \mca_\ma\rp$
\end{enumerate}

The elements of each $\ap_\alpha$ can be naturally partially ordered
(by inclusion in both coordinates).

\begin{conv} \label{c:fil}
For $\ma \in \ap_\alpha$, denote by $\fil\lp\de \cup \ma\rp$ the
filter on $I$ generated by $\de \cup \mca_\ma$.
\end{conv}

\br \step{2. Aim of the inductive step.} Our aim in the key
inductive step will be to prove the following. Note that from a
certain point of view, this is about the
Boolean algebra $\mcp(I)/\de$; however, it clarifies our
presentation here to refer explicitly to $I$ and to the given types
as they are presented in the reduced product. By quantifier elimination, a type here is a choice 
of $\lambda$ parameters and a function from $\lambda$ to $2$. 

\begin{claim} \label{p-claim} If $(A)$ then $(B)$:
\begin{enumerate}
\item[$(A)$]
\begin{enumerate}
\item $\ma \in \ap_{\alpha}$ 
\item $M \models \trg$
\item $h_i \in {^I M}$ for $i < \lambda$ 
\item $\eta \in {^\lambda 2}$
\item For $ i < j < \lambda$, if $\eta(i) = \eta(j)$ then $A_{i,j} = I$, and if $\eta(i) \neq
\eta(j)$ let
\[ A_{i, j} = \{ t \in I : h_i(t) \neq h_j(t) \} = I \mod \fil\lp\de \cup \ma\rp \] 
\item $p = \{ (xR(h_i/\de_1))^{[\eta(i)]} ~ : ~ i < \lambda \}$ is a type in $M^I/\de_1$ for every ultrafilter $\de_1 \supseteq \fil\lp\de \cup \ma\rp$
\end{enumerate}

\br
\item[$(B)$] There are $\mb, \overline{B}$ such that:
\begin{enumerate}
\item $\ma \leq \mb \in \ap_{\alpha + 1}$
\item $\overline{B} = \langle B_i : i < \lambda \rangle$
\item $B_i \in \fil\lp\de \cup \mb\rp$
\item $B_i \cap B_j \subseteq A_{i, j} \mod \de$ for $i < j < \lambda$ such that
$\eta(i) \neq \eta(j)$
\end{enumerate}
\end{enumerate}
\end{claim}

\br The proof will follow from Steps 11-13 below, following some
intermediate definitions.  In Step 13, we verify that Claim \ref{p-claim}(B) is sufficient to realize the type. 

\br
\begin{defn} \label{d:s-mult}
Given any sequence $\overline{B} = \langle B_i : i < \lambda
\rangle$ from Claim \ref{p-claim}(B), let the sequence
$\overline{B}_* = \langle B_u : u \in \lao \rangle$ be given by:
$B_{\{i\}} = B_i$ for $i < \lambda$, and $B_u = \bigcap \{ B_{\{i\}}
: i \in u \}$ for $u \in \lao$, $|u| \geq 2$.
\end{defn}

\br \noindent\emph{Discussion.} Condition $(A)$ corresponds to the
data of a random graph type over the parameters $\{ h_i : i <
\lambda \}$. Define a sequence $\overline{A} = \langle A_u : u \in
\lao \rangle$ by $|u| = 1$ implies $A_u = I$, $u=\{i,j\}$, $i \neq
j$ implies $A_u = A_{i,j}$ as defined above, and $|u| \geq 3$
implies $A_u = \bigcap \{ A_v : v \subseteq u, |v| = 2 \}$. Then
$\overline{A}$ gives a distribution for this type.

Note that $t \in A_{i,j}$ implies $\{ (xR(h_i/\de_1))^{[\eta(i)]},
(xR(h_j/\de_1))^{[\eta(j)]} \}$ is consistent. Since random graph
types have 2-compactness, and $(\exists x)xR(h_i/\de_1)$ is sent to
$I$ by the \los map, item $(B)(d)$ implies that the corresponding
sequence $\overline{B}_* = \langle B_u : u \in \lao \rangle$ from
Definition \ref{d:s-mult} refines $\overline{A}$ modulo $\de$ and
is multiplicative $\mod \de$. This will be sufficient to realize
the type given the background assumption of excellence, see Step 13
below.

\br

\br\noindent\textbf{Data for the key inductive step}. Steps 3-10 below are
in the context of Claim \ref{p-claim}, meaning that we fix the data
of Claim \ref{p-claim}(A) and define the following objects based on
it. Thus all objects defined here are implicitly subscripted by
$\alpha$ and $\ma$ and depend on the choice of sequence $\langle A_i
: i < \lambda \rangle$ at this inductive step.

\br

\step{3. The set $\eff^1_i = \eff^1_{\alpha, \ma, i}$.} For each $i
< \lambda$,
define $\eff^1_{i}$ to be the set of all $f \in \fin(\gee_\ma)$ such that for some $j \leq i$: 
\begin{enumerate}
\item $h_j |_{A_f} = h_i |_{A_f} \mod \de$
\item $\gamma < j \implies h_\gamma |_{A_f} \neq h_i|_{A_f} \mod \de$ 
\\ i.e., $\gamma < j \implies$
\[ \{ t \in A_f : h_j(t) \neq h_i(t) ~\mbox{or}~ h_\gamma(t) = h_i(t) \} = \emptyset \mod \de \]
\end{enumerate}

Note that fixing $A_f$ there is a minimal $j$ such that $h_j|_{A_f}
= h_i|_{A_f}$ $\mod \de$. Since the ordinals are well ordered, we can choose a
least $j$ for which there is such a witness $A_f$.

\step{4. The set $\eff^2_i = \eff^2_{\alpha, \ma, i}$.} For each $i
< \lambda$, choose $\eff^2_{i}$ so that:
\begin{enumerate}
\item $\eff^2_{i} \subseteq \eff^1_{i}$
\item $f^\prime \neq f^{\prime\prime} \in \eff^2_{i} \implies f^\prime, f^{\prime\prime}$ are incompatible functions
\item $\eff^2_{i}$ is maximal under these restrictions
\end{enumerate}

\step{5. Density.} Notice that for $\ell = 1,2$ $\eff^\ell_{i}$ is
pre-dense, i.e. for every $f^\prime \in \fin(\gee_\ma)$ for some
$f^{\prime\prime} \in \eff^1_{i}$ the functions $f^\prime,
f^{\prime\prime}$ are compatible.

The set $\eff^1_{i}$ is dense, meaning that for each $f^\prime \in
\fin(\eff_{\ma})$ there is $f^{\prime\prime} \in \eff^1_{i}$ such
that $f^\prime \subseteq f^{\prime\prime}$. Moreover, it is open,
meaning that if $f^{\prime\prime} \in \eff^1_{i}$ and
$f^{\prime\prime} \subseteq f^{\prime\prime\prime}$  then
$f^{\prime\prime\prime} \in \eff^1_{i}$. Here by ``$\subseteq$'' we
mean that the domain of the smaller function is contained in the
domain of the larger function, and the two functions agree on their
common domain. [Note that $f \subseteq f^\prime \implies A_{f^\prime} \subseteq A_f$.]

\step{6. The collision function.}
For each $i < \lambda$ define the function $\rho_{i} : \eff^2_{i} \rightarrow i$ 
by:
\[   \rho_{i}(f) = \operatorname{min} \{ j \leq i : h_i|_{A_f} = h_j|_{A_f} \mod \de \} \]

Note that by the definition of $\eff^2_i$, this is the whole story
in the sense that if $\rho_i(f) = j$ then for no $f^\prime \supseteq
f$ does there exist $j^\prime < j$ such that $h_i|_{A_{f^\prime}} =
h_{j^\prime}|_{A_{f^\prime}} \mod \de$.

\step{7. The new support.} By induction on $i < \lambda$ choose $g_i
\in \gee_* \setminus \gee_\ma \setminus \{ g_j : j < i \}$ (we make no further requirements
on the sequence, but note the functions will be distinct).

\step{8. The partition.}
For each $i<\lambda$, the sequence $\langle A_f : f \in \eff^2_{i}
\rangle$ is (by definition) a sequence of pairwise disjoint sets. 

\step{9. The refinement.} Let $\overline{B} =
\langle B_i : i < \lambda \rangle$ where for $i<\lambda$ 
\[  B_i = \bigcup \{ A_f \cap \left( g^{-1}_{\rho_{i}(f)}(\{ 0 \})  \right)^{[\eta(i)]} : f \in \eff^2_{i} \} \]
where recall that for $B \subseteq I$, $B^{[1]} = B, B^{[0]} = I
\setminus B$.

\step{10. Definition of $\mb$.} Finally, we define
\[ \mb = \left( \gee_\ma \cup \{ g_i : i < \lambda \}, \mca_\ma \cup \{ B_i : i < \lambda \} \right) \]

Note that in contrast to the ``usual'' construction, here we have addressed the problem of realizing a type 
by using $\lambda$ functions of range $\aleph_0$, rather than a single function of range $\lambda$. Moreover,
we do not complete to a good triple at the end of the inductive step.

\br

\step{11. Proof of Claim \ref{p-claim}(B)(a).} We need to check that
$\mb$ as defined in Step 10 satisfies clauses (a)-(e) of the
definition of approximation from Step 1. The only non-trivial part
is proving that $\emptyset \notin \fil\lp\de \cup \mca_\mb\rp$,
recalling Convention \ref{c:fil}.

Suppose we are given $C_0, \dots C_{k-1} \in \mca_\ma$ and $i_0,
\dots i_{n-1} < \lambda$. It will suffice to prove that
\[ \bigcap_{j < n} B_{i_j} \cap \bigcap_{\ell < k} C_\ell \neq \emptyset \mod \de \]

Informally speaking, we first try to find a $\de$-nonzero set on
which the corresponding parameters $h_{i_j}$ are distinct. On such a
set, the instructions for each $B_{i_j}$ are clearly compatible, so
we can then find some $A_{f_*} \in \fin_s(\gee_*)$, Definition \ref{good-triples}, which is
contained in their intersection $\mod \de$. We now give the
details.

First, by $(A)(e)$ of the inductive step and the fact that
$\fil\lp\de \cup \mca_\ma\rp$ is a filter, there is 
$A \in \fil\lp\de\cup \ma\rp$, $A \subseteq \bigcap \{ C_\ell : \ell < k \}$ such that
$j < j^\prime < n \land \eta(j) \neq \eta(j^\prime) \implies A_{j,
j^\prime} \supseteq A$. Moreover, by the definition of
approximation, any $A \in \fil\lp\de \cup \ma\rp$ contains a set which
is supported by $\gee_\ma$. Let $f \in \fin(\gee_\ma)$ be such that
$A_f \subseteq A \mod\de$.

Second, recall that each $\eff^2_i$ is pre-dense. So for each $j<n$,
we may choose $f_{i_j} \in \eff^2_{i_j}$ which is compatible with
$f$. As we can increase $f$, without loss of generality, for $j < n$
there is $f_{i_j} \in \eff^2_{i_j}$ such that $f_{i_j} \subseteq f$
(choose these by induction on $j<n$). By choice of $f$, for no $j
<j^\prime < n$ is it the case that $h_{i_j} = h_{i_{j^\prime}}$ on
$A_f$. In other words,
\[ j < j^\prime < n \land \eta(j) \neq \eta(j^\prime) \implies \rho_{i_j}(f_{i_j}) \neq \rho_{i_{j^\prime}}(f_{i_{j^\prime}}) \]
Thus (identifying functions with their graphs) the function $f_*$
defined by
\[ f_* = f \cup \bigcup \{ f_{i_j} : j < n \} \cup \{ (g_{k_j}, t_{j}) : k_j = \rho_{i_j}(f{i_j}), t_{j} = \eta(i_j) \} \]
is indeed a function, thus an element of $\fin(\gee_*)$. Clearly
$A_{f_*} \subseteq B_{i_j}$ for each $j<n$, and $A_{f_*} \neq
\emptyset \mod \de$ by the hypothesis that $(I, \de, \gee_*)$ is
a good triple. This completes the proof of Step 11.

\br

Now Claim \ref{p-claim}(B)(b)-(c) obviously hold, so we are left
with:

\step{12. Proof of Claim \ref{p-claim}(B)(d).} We now show that if
$i \neq j < \lambda$, $\eta(i) \neq \eta(j)$ then $B_i \cap B_j
\subseteq A_{i,j} \mod \de$. Note that if $\eta(i) = \eta(j)$ the
inclusion holds trivially, which is why we assume $\eta(i) \neq
\eta(j)$ (so $i \neq j$).

Assume for a contradiction that $A_* := (B_i \cap B_j) \setminus
A_{i,j} \neq \emptyset \mod \de$. As $B_i, B_j, A_{i,j}$ are
supported by $\gee_\mb$, there is $f_* \in \fin(\gee_\mb)$ such that
$A_{f_*} \subseteq A_*$ $\mod \de$.

As there is no problem increasing $f_*$, we may choose $f_i \in
\eff^2_i$, $f_j \in \eff^2_j$ such that $f_i \subseteq f_* \land f_j
\subseteq f_*$. In other words, $A_{f_*} \subseteq A_{f_i} \cap
A_{f_j}$. Recall that
\[ A_{i, j} = \{ t \in I : h_i(t) \neq h_j(t) \} \]
By the definition of $A_{i,j}$, since $A_{f_*} \cap A_{i,j} =
\emptyset \mod \de$ it must be that $h_i = h_j$ on $A_{f_*} \mod
\de$. Because we chose $f_i$ and $f_j$ from $\eff^2_i$ and
$\eff^2_j$, respectively, there must be $k_i \leq i$ and $k_j \leq
j$ so that $h_i = h_{k_i}$ on $A_{f_i}$, and $h_j = h_{k_j}$ on
$A_{f_j}$. Since equality is transitive, there is $k \leq
\min\{i,j\}$ such that $h_i = h_j = h_k$ on $A_{f_*}$. In the
notation of Step 6, $\rho_i(f_i) = \rho_j(f_j) = k$.

Now we look at the definition of $B_i, B_j$. Since $\eta(i) = \trv_i
\neq \eta(j) = \trv_j$, we have that $B_i \cap A_{f_*} \subseteq
(g^{-1}_k(0))^{\trv_i}$  whereas $B_j \cap A_{f_*} \subseteq
(g^{-1}_k(0))^{\trv_j}$. Thus $B_i \cap B_j \cap A_{f_*} =
\emptyset$, which is the desired contradiction.

\step{13. Finishing the key inductive step.} Why is Claim
\ref{p-claim} sufficient to realize the type? Let us rephrase the
problem as follows.

\begin{cor} \label{p-cor}
Let $\de$ be the excellent background filter defined at the
beginning of the proof.
Let $\ma$, $\langle h_i : i < \lambda \rangle$, $\langle A_{i,j} : i
< j < \lambda \rangle$, $\eta$, $p$ be as given in Claim
$\ref{p-claim}(B)$ at the inductive stage $\alpha$.
Let $\langle B_i : i <\lambda \rangle$, $\mb$ be as given by Claim
$\ref{p-claim}(A)$. 

Then in any ultrafilter $\de_1$ extending
$\fil\lp\de \cup \mb\rp$, the type $p$ is realized.
\end{cor}

\begin{proof}
Let $\langle A_u : u \in \lao \rangle$ be the sequence defined in
the Discussion in Step 2, which corresponds to a distribution of the
type $p$. Let $\overline{B}_* = \langle B_u : u \in \lao \rangle$ be
constructed from $\{ B_i : i < \lambda \}$   
as in Definition \ref{d:s-mult} above.
By definition $\overline{B}_*$ is multiplicative. In step 12, it was
shown that 
$\overline{B}_*$ refines 
$\overline{A} \mod \de$.

Define a third sequence $\overline{B}^\prime = \langle B^\prime_u :
u \in \lao \rangle$ by $B^\prime_u = B_u \cap A_u$ for $u \in \lao$.
Then the sequence $\overline{B}^\prime$ truly refines
$\overline{A}$, but is only multiplicative modulo $\de$.

By the excellence of $\de$ [i.e. Fact \ref{e:fact} above] we may replace $\overline{B}^\prime$
by a sequence $\langle B^{**}_u : u \in \lao \rangle$ such that:
\begin{itemize}
\item $u \in \lao$ implies $B^{**}_u \subseteq B^\prime_u \subseteq A_u$
\item $u \in \lao$ implies $B^{**}_u = B^\prime_u \mod \de$, thus $B^{**}_u \in \fil\lp\de \cup \mb\rp$
\item $\langle B^{**}_u : u \in \lao \rangle$ is multiplicative
\end{itemize}

Thus we may realize the type. 
\end{proof}

\step{Step 14: Adding subsets of the index set.}

\begin{claim} \label{p-set}
Let $\alpha < 2^\lambda$, $\ma \in \ap_\alpha$, $A \subseteq I$ be
given. Then there is $\mb \in \ap_{\alpha + 1}$ such that either $A
\in \fil\lp\de \cup \mb\rp$ or $I \setminus A \in \fil\lp\de \cup \mb\rp$.
\end{claim}

\begin{proof}
Let $X = A$ if $\emptyset \notin \fil\lp\de \cup \ma \cup \{A\}\rp$,
otherwise let $X = I \setminus A$. By the hypothesis that $(I,
\de, \gee_*)$ is $(\lambda, \mu)$-good, we may choose a partition
$\langle A_{f_i} : i < \mu \rangle$ of $\mcp(I)/\de$ supporting $X$ such that each
$f_i \in \fin(\gee_*)$. Let $\gee_\mb = \gee_\ma \cup \bigcup \{
\dm(f_i) : i < \mu \}$, and let $\mca_\mb = \mca_\ma \cup \{ X \}$.
This suffices.
\end{proof}

\br \noindent\textbf{Proof of Theorem \ref{t:rg}.} We now prove the
theorem.

\begin{proof} (of Theorem \ref{t:rg})
Without loss of generality, $|M| \leq 2^\lambda$.

Let $\langle C_\alpha : \alpha < 2^\lambda \rangle$ enumerate
$\mcp(I)$. Let $\langle \overline{h}^\alpha, \eta^\alpha : \alpha < 2^\lambda
\rangle$ enumerate all $\eta \in {^\lambda 2}$ and 
all $\overline{h} = \langle h_i : i < \lambda \rangle$ with $h_i \in {^IM}$, 
with each such pair appearing $2^\lambda$ times in the enumeration.

We build the ultrafilter by induction on $\alpha \leq 2^\lambda$.
That is, we choose $\ma_\alpha \in \ap_\alpha$ by induction on
$\alpha < 2^\lambda$ such that:
\begin{enumerate}
\item $\beta < \alpha \implies \ma_\beta \leq \ma_\alpha$
\item if $\alpha = 2\beta + 1$ then either $C_\beta \in \fil\lp\de \cup \ma_\alpha\rp$
or $I \setminus C_\beta \in \fil\lp\de \cup \ma_\alpha\rp$
\item if $\alpha = 2\beta + 2$ and $(\ma_{2\beta + 1}, \overline{h}^\beta, \eta^\beta)$ satisfies
Claim \ref{p-claim}$(A)$, then there is $\overline{B}$ which, along
with $\ma_\alpha$, satisfies Claim \ref{p-claim}$(B)$.
\item if $\alpha$ is a limit ordinal then $\ma_\alpha$ is the least upper bound of
$\{ \ma_\beta : \beta < \alpha \}$ in the natural partial order,
i.e. given by taking the union in both coordinates.
\end{enumerate}

The odd inductive steps are given by Claim \ref{p-set} above, and
clearly ensure that $\fil\lp\de \cup \ma_{2^\lambda}\rp$ is an
ultrafilter. The even inductive steps are given by Claim \ref{p-claim}
above, and Corollary \ref{p-cor} ensures saturation.

Thus letting $\de_1 = \fil\lp\de \cup \ma_{2^\lambda}\rp$, we complete
the proof.
\end{proof}

\br
\br
\section*{Appendix}

We include some definitions used mainly in \S \ref{intro:tr}. For details, see \cite{MiSh:996} \S 1. 

\begin{defn} \label{mu-defn} \emph{(Shelah \cite{Sh:c} Definition III.3.5)}
Let $\de$ be an ultrafilter on $\lambda$.
\[  \mu(\de) :=  \operatorname{min} \left\{ \rule{0pt}{15pt}
\prod_{t<\lambda}~ n_t /\de  ~: ~ n_t < \aleph_0, ~\prod_{t<\lambda}
~n_t/\de \geq \aleph_0 \right\} \]
be the minimum value of the product of an unbounded sequence of cardinals
mod $\de$.
\end{defn}

\begin{defn} \emph{(Good for equality, Malliaris \cite{mm5})}
Let $\de$ be a regular ultrafilter. Say that $\de$ is \emph{good for equality}
if for any set $X \subseteq N = M^I/\de$,
$|X| \leq |I|$, there is a distribution $d: X \rightarrow \de$ such that $t \in \lambda, t \in d(a) \cap d(b)$ implies that 
$(M \models a[t] = b[t]) \iff (N \models a = b)$.
\end{defn}

\begin{defn} \emph{(Lower cofinality, $\lcf(\kappa, D)$)} \label{lcf}
Let $D$ be an ultrafilter on $I$ and $\kappa$ a cardinal. Let $N = (\kappa, <)^I/\de$.
Let $X \subset N$ be the set of elements above the diagonal embedding of $\kappa$.
We define $\lcf(\kappa, D)$ to be the cofinality of $X$ considered with the reverse order.
\end{defn}

\begin{defn} \emph{(Good ultrafilters, Keisler \cite{keisler-1})} 
\label{good-filters}
The filter $\de$ on $I$ is said to be \emph{$\mu^+$-good} if every $f: \fss(\mu) \rightarrow \de$ has
a multiplicative refinement, where this means that for some $f^\prime : \fss(\mu) \rightarrow \de$,
$u \in \fss(\mu) \implies f^\prime(u) \subseteq f(u)$, and $u,v \in \fss(\mu) \implies
f^\prime(u) \cap f^\prime(v) = f^\prime(u \cup v)$.

Note that we may assume the functions $f$ are monotonic.

$\de$ is said to be \emph{good} if it is $|I|^+$-good.
\end{defn}

\begin{rmk} \label{r:good}
The importance of good filters here arises from Keisler's observation that when $\de$ is regular, $A \subseteq N := M^\lambda/\de$, $p \in S(A)$, 
$|A| \leq \lambda$, 
$p = \{ \vp_i(x;\overline{a}_i) : i < \lambda \}$
then $p$ is realized in $N$ iff there is some multiplicative $f: \fss(\lambda) \rightarrow \de$ which refines the ``existential'' map
$ \sigma \mapsto \{ t \in \lambda ~:~ M \models \exists x \bigwedge_{i \in \sigma} \vp_i(x;\overline{a}_i[t]) \}$.
\end{rmk}

\end{document}